\documentclass[a4paper,10pt,leqno]{article}

\usepackage{amssymb,amsmath}
\usepackage[abbrev]{amsrefs}
\usepackage{xy}
\usepackage{color}
\xyoption{all}

\numberwithin{equation}{subsection}
\newtheorem{defn}{Definition}[section]
\newtheorem{corollary}[defn]{Corollary}
\newtheorem{rem}[defn]{Remark}
\newtheorem{exm}[defn]{Example}
\newtheorem{lemma}[defn]{Lemma}
\newtheorem{theorem}[defn]{Theorem}
\newtheorem{xproof}{{\it Proof. }}

\newenvironment{definition}{\begin{defn}\em}{\end{defn}}
\newenvironment{remark}{\begin{rem}\em}{\end{rem}}
\newenvironment{example}{\begin{exm}\em}{\end{exm}}
\newenvironment{proof}{\begin{xproof}\em}{\end{xproof}}

\def\qed{\hspace{0.3cm}{\rule{1ex}{2ex}}}
\newcommand\V{\bigvee}
\newcommand\ie{i.e.}
\newcommand\eg{e.g.}
\newcommand\Rel{\operatorname{Rel}}
\newcommand\cf{\textrm{cf.}}
\newcommand\opens{\operatorname{\mathcal{O}}}
\newcommand\spp{\operatorname{\varsigma}}
\newcommand\downsegment{\operatorname{\downarrow}}
\newcommand\Frm{\textit{Frm}}
\newcommand\Loc{\textit{Loc}}
\newcommand\Top{\textit{Top}}
\newcommand\rs{\operatorname{R}}
\newcommand\ls{\operatorname{L}}
\newcommand\ts{\operatorname{T}}
\newcommand\ident{\mathrm{id}}
\newcommand\ipi{\operatorname{\mathcal I}}
\newcommand\SL{\mathit{SL}}
\newcommand\ssq{\mathit{StabQu}}
\newcommand\iqf{\mathit{InvQuF}}
\newcommand\quantales{\mathit{Qu}}
\newcommand\sppQu{\quantales^{\spp}}
\newcommand\StabQu{\mathit{StabQu}}
\newcommand\eqsppQu{\quantales^{\textrm{eq}}}
\newcommand\ueqsppQu{\quantales_e^{\textrm{eq}}}
\newcommand\steqsppQu{\quantales_1^{\textrm{eq}}}
\newcommand\usppQu{\quantales_e^{\spp}}
\newcommand\stsppQu{\quantales_1^{\spp}}
\newcommand\sppuQu{\quantales_{(e)}^{\spp}}
\newcommand\basedQu{\quantales^{\flat}}
\newcommand\stbasedQu{\quantales_1^{\flat}}
\newcommand\uquantales{\quantales_e}
\newcommand\stquantales{\quantales_1}
\newcommand\intg[1]{\widetilde{#1}}
\newcommand\lres[2]{{_{#1}{\vert}{#2}}}
\newcommand\rres[2]{{{#1}{\vert}_{#2}}}

\begin{document}

\title{Effective equivalence relations and principal quantales}

\author{Juan Pablo Quijano and Pedro Resende\thanks{Work funded by FCT/Portugal through the LisMath program and projects EXCL/MAT-GEO/0222/2012 and PEst-OE/EEI/LA0009/2013.}}

\date{~}

\maketitle

\vspace*{-1cm}

\begin{abstract}
Stably supported quantales generalize pseudogroups and provide an algebraic context in which to study the correspondences between inverse semigroups and \'etale groupoids. Here we study a further generalization where a non-unital version of supported quantale carries the algebraic content of such correspondences to the setting of open groupoids. A notion of principal quantale is introduced which, in the case of groupoid quantales, corresponds precisely to effective equivalence relations.
\\
\vspace*{-2mm}~\\
\textit{Keywords:} Open localic groupoids, effective equivalence relations, supported quantales.\\
\vspace*{-2mm}~\\
2010 \textit{Mathematics Subject
Classification}: 06F07 (06D22; 20L05; 20M18; 22A22; 54B30)
\end{abstract}

\tableofcontents


\section{Introduction}\label{introduction}

Let $G$ be a topological group. Its topology $\opens(G)$ carries a natural semigroup structure whose multiplication is the pointwise product of open sets:
\[
UV = \{gh\mid g,h\in G\}\;.
\]
More generally, let $G$ be a topological groupoid:
\begin{equation*}
\xymatrix{
G\ =\ G_2\ar[r]^-{m}& G_1\ar@(ru,lu)[]_i\ar@<1.2ex>[rr]^r\ar@<-1.2ex>[rr]_d&&G_0\ar[ll]|u
}\;.
\end{equation*}
Here $G_0$ is the \emph{object space}; $G_1$ is the \emph{arrow space}; $d$ and $r$ are, respectively, the continuous \emph{domain} and \emph{range maps}; $u$ is the continuous \emph{unit arrows map}; $i$ is the \emph{inversion homeomorphism}; $m$ is the continuous \emph{multiplication map};
and $G_2$ is the subspace of the product space $G_1\times G_1$ that consists of the \emph{composable pairs} of arrows; that is,
\[
G_2=\bigl\{(g,h)\in G_1\times G_1\mid r(g)=d(h)\bigr\}\;.
\]
Then the topology $\opens(G_1)$ carries a semigroup structure if $\opens(G_1)$ is closed under the pointwise product of open sets, which now is given by
\[
UV=\bigl\{m(g,h)\mid (g,h)\in G_2\bigr\}\;.
\]
Equivalently, the multiplication is well defined precisely for those groupoids whose domain map $d$ is open (which is trivially true for a topological group). These are called \emph{open} groupoids.

Such semigroups encode plenty of information about open groupoids \cite{Re07,PR12}, hence yielding algebraic tools with which to study geometric structures that give rise to them. One particular class of examples which is worthy of special mention is that of \emph{{\'e}tale} groupoids (those such that $d$ is a local homeomorphism), for which $u:G_0\to G_1$ is an open map (so, for topological groups, being \'etale means being discrete). Hence, for an \'etale groupoid $G$ the semigroup $\opens(G_1)$ is a monoid whose unit is the image $u(G_0)$. In fact, the openness of $u$ is equivalent, for any open groupoid $G$, to $G$ being \'etale, so, in essence, the open groupoids are the topological groupoids whose topologies are semigroups, and the \'etale groupoids are the topological groupoids whose topologies are monoids.

More precisely, the correspondence between semigroups and groupoids also takes into account a complete lattice structure (the order is given by inclusion of open sets) that turns the semigroups into quantales. In the case of \'etale groupoids these are unital quantales known as inverse quantal frames, and they relate closely to inverse semigroups because the category of pseudogroups (complete and infinitely distributive inverse monoids) is equivalent to the category of inverse quantal frames. This is proved in~\cite{Re07}, where the relations to \'etale groupoids, both topological and localic, are established precisely and include a bijection, up to isomorphisms, between the class of localic \'etale groupoids and that of inverse quantal frames. A consequence of these is that the role played by inverse
semigroups in relation to topological \'etale groupoids (see, \eg, \cites{Kumjian84,RenaultLNMath,Paterson}) is subsumed by inverse quantal frames (this is further generalized for non-involutive quantales and \'etale categories in~\cite{KL}), and in \cite{MaRe10} it has been shown that 
this generalizes classical topological correspondences between inverse semigroups and
germ groupoids (see also \cite{LL}).

Given such results for \'etale groupoids one is naturally led to asking whether open non-\'etale groupoids relate to their quantales in an equally well behaved way. This question has been addressed in~\cite{PR12}, where technical difficulties were identified which do not exist in the case of \'etale groupoids. In the present paper we propose to address this correspondence again, from a  perspective that is already present in~\cite{Re07} but not in~\cite{PR12}. In order to explain this let us begin by recalling one particular aspect of the theory of inverse quantal frames developed in \cite{Re07}, namely that such quantales are introduced as instances of a more general class, that of stably supported quantales, which has nice algebraic properties and is interesting in its own right. A supported quantale, and specifically one whose support is stable, can be regarded as an abstract generalization (more general than modular quantales) of the unital involutive quantale of binary relations $\wp(X\times X)$ on a set $X$. For instance the applications to propositional normal modal logics in \cite{MarcR} are based on this idea and show that stably supported quantales provide useful ``Lindenbaum algebras'' that not only take  the algebra of propositions into account but also the algebra of accessibility relations, thus providing a full algebraicization of such logics with which, in particular, constructive versions of completeness theorems can be proved. One advantage of having a notion of supported quantale that caters for open groupoids is therefore the possibility of carrying the semantics of such logics to examples such as the holonomy groupoid of a foliated manifold (see, \eg, \cite{MoerMrcunBook}), without having to resort to a Morita equivalent \'etale groupoid.

The algebraic characterization of the class of quantales that corresponds to open groupoids given in \cite{PR12} is a direct generalization of inverse quantal frames and no corresponding generalization of supported quantales is provided. The purpose of this paper is to address the correspondence between quantales and non-\'etale groupoids in a way that recovers some of the algebraic simplicity and convenience of supported quantales. Technically, we work with involutive quantales (not necessarily unital) which are also $A$-$A$-bimodules satisfying suitable conditions, where $A$ is a locale playing a role similar to that of the unit space of a groupoid. Then a general support $\spp:Q\to A$ is defined to be a sup-lattice homomorphism satisfying conditions that mimick those of \cite{Re07}. Due to the absence of the unit for the quantale multiplication we shall see that no naive generalization of stable supports is enough to obtain a theory with the same good properties of the unital case. For that reason we introduce the stronger notion of \emph{equivariantly supported $A$-$A$-quantale}. This will be fully developed in section 3 and then it will be used all over the rest of the paper.

Then section 4 addresses these generalized supported quantales when they are also locales, \ie, quantal frames. The main aim is not only to complete the toolbox needed in order to describe the quantales of open groupoids but also to develop the theory of $A$-$A$-quantal frames on its own. In particular, we shall define \emph{principal quantales}. Technically these are equivariantly supported $A$-$A$-quantal frames such that $Q\cong \rs(Q)\otimes_{\ts(Q)} \ls(Q)$, where $\rs(Q)$ and $\ls(Q)$ are the subquantales of right and left sided elements, respectively, and $\ts(Q)=\rs(Q)\cap \ls(Q)$. This notion will play a central role in this paper and will be addressed again in section 5. The notion of \emph{reflexive quantal frame} is also introduced, aiming to make up for the loss of the multiplicative unit of our quantales. A reflexive quantal frame is an $A$-$A$-quantal frame equipped with a suitable locale homomorphism $\upsilon:Q\to A$ that in the unital case is the one given by $\upsilon(q)=q\wedge e$. Finally, similarly to a part of \cite{PR12}, we use the multiplicativity axiom, which is automatically satisfied by inverse quantal frames. The rest of section 4 deals with the study of properties of equivariantly reflexive supported $A$-$A$-quantal frames satisfying the multiplicativity axiom.

In section 5 we reframe the work of \cite{PR12} in the language of non-unital supported quantales. This provides a natural algebraic description of general open groupoids, which we apply elsewhere~\cite{QuijanoPhD}*{Chapter 6} in order to extend the correspondence between groupoid sheaves and quantale sheaves that was carried out in \cite{GSQS} for \'etale groupoids and inverse quantal frames. We also expect that this theory will have applications to specific examples of non-\'etale groupoids such as the \'etale-complete groupoids of \cites{KoMoer,Moer88,Moer90,Bunge}, which in the present paper will surface in the specific form of effective equivalence relations. In order to obtain the envisaged algebraic description of open groupoids we shall begin by introducing two independent axioms. The first one concerns \emph{unit laws}, which together with the axioms of the previous sections give us the following result: if $(Q,\spp,\upsilon)$ is a multiplicative equivariantly supported reflexive $A$-$A$-quantale frame that satisfies the unit laws, and if $G$ is its associated involutive localic graph,
then $G$ is an open involutive category (see Lemma~\ref{opencategory}) where $Q=\opens(G_1)$ and $A=\opens(G_0)$. We note that in the unital case such laws are not required because the map $\upsilon$ is an open map of locales. The second axiom expresses \emph{inverse laws}, and it will be introduced in order to define the notion of \emph{groupoid quantale}, by which is meant a multiplicative equivariantly reflexive $A$-$A$-quantale frame $(Q,\spp,\upsilon)$ satisfying the unit laws and the inverse laws. This leads to our main result, namely: if $(Q,\spp,\upsilon)$ is a groupoid quantale  then its involutive localic graph $G$ is an open groupoid (see Theorem~\ref{opengroupoid}). In the unital case a quantal frame $Q$ satisfies the inverse laws if and only if it it is covered by its partial units: $\V \ipi(Q)=1_Q$. These two axioms provide us with the most perspicuous non-unital generalization of inverse quantal frames so far.

The last part of this paper addresses two specific examples of quantales of non-\'etale groupoids. The first example is the quantale of the ``pair groupoid'' of an open groupoid (see section \ref{pairgroupoid}). Then we shall revisit the notion of principal quantale which, in the case of groupoid quantales, will be seen to correspond precisely to effective equivalence relations (see section \ref{subsec: principal groupoids}). In particular, this gives us a first example of the quantale of an \'etale-complete groupoid in a simplified situation, namely when the topos $BG$ is localic (see Corollary~\ref{cor: etalecompleteness}).


\section{Preliminaries}

The purpose of this section is to recall some definitions and to set up notation and terminology.


\subsection{Quantales}

By an \emph{involutive quantale} is meant a sup-lattice $Q$ equipped with an associative \emph{multiplication} $(a,b)\mapsto ab$ satisfying
\[
a\bigl(\V S\bigr) = \V_{b\in S} ab\;,\quad\bigl(\V S\bigr) a = \V_{b\in S} ba
\]
for all $a\in Q$ and $S\subset Q$, and with a monotone semigroup \emph{involution} $a\mapsto a^*$. Note that $1_Q=\V Q$ and $0_Q=\V\emptyset$ and an involution is necessarily an order isomorphism, and thus it preserves joins. If $Q$ has a multiplication unit $e_Q$ (or just $e$) that makes it a monoid then we say that $Q$ is a \emph{unital involutive quantale}. We denote by $\rs(Q)$, $\ls(Q)$, and $\ts(Q)$, respectively,  the subquantales of right sided, left sided, and two sided elements of $Q$:
\begin{align*}
\rs(Q) &=\{a\in Q\mid a1\le a\}\;,\\
\ls(Q) &=(\rs(Q))^*\;,\\
\ts(Q) &= \rs(Q)\cap \ls(Q)\;.
\end{align*}

A \emph{homomorphism} of involutive quantales is a homomorphism $f:Q\to R$ of involutive semigroups that is also a homomorphism of sup-lattices; that is, for all $S\subset Q$ and $a,b\in Q$ we have
\[
f\bigl(\V S\bigr)=\V_{a\in S} f(a)\;,\quad f(ab)=f(a)f(b)\;,\quad f(a^*)=f(a)^*\;.
\]
The category of involutive quantales and their homomorphisms will be denoted by $\quantales$. The subcategory with the same objects and whose homomorphisms $f:Q\to R$ are \emph{strong}, \ie, such that $f(1_Q)=1_R$, is denoted by $\stquantales$.
A homomorphism of unital involutive quantales is \emph{unital} if it is a monoid homomorphism. The subcategory of $\quantales$ that consists of the unital quantales and unital homomorphisms is denoted by $\uquantales$.


\subsection{Locales and groupoids}

We denote by $\Loc$ the category of \emph{locales}, which is the dual of the category $\Frm$ of frames and frame homomorphisms. The arrows of $\Loc$ are referred to as \emph{continuous maps}, or simply \emph{maps}, of locales. If $X$ is a locale we shall usually write $\opens(X)$ for the same locale regarded as an object of $\Frm$ (an imitation of the notation for the topology of a topological space $X$). 
If $f:X\to Y$ is a map of locales we shall refer to the frame homomorphism $f^*:\opens(Y)\to \opens(X)$ that defines it as its \emph{inverse image homomorphism}. A map of locales $f: L\to M$ is said to be \emph{semiopen} if $f^*:\opens(M)\to \opens(L)$ preserves all meets (or, equivalently if it has a left adjoint $f_!:\opens(X)\to \opens(Y)$ called the \emph{direct image}), and, it is said to be \emph{open} if it is semiopen and satisfies the so-called \emph{the Frobenius reciprocity condition}:
\begin{equation}\label{preliminaries, eq: frobeniuscondition}
f_!(x\wedge f^*(y)) = f_!(x)\wedge y\,,
\end{equation}
for all $x\in \opens(X)$ and $y\in \opens(Y)$. The product of $X$ and $Y$ in $\Loc$ is $X\times Y$. it coincides with the tensor product $\opens(X)\otimes\opens(Y)$ in the \emph{category of sup-lattices} $\SL$, so we write $\opens(X\times Y)=\opens(X)\otimes \opens(Y)$.

An \emph{internal groupoid} $G$ in a category $C$ (with enough pullbacks) consists of objects $G_0$ and $G_1$ of $C$, of \emph{objects} and \emph{arrows}, respectively, equipped with the following morphisms in $C$ satisfying the usual axioms of an internal category plus the inverse laws of a groupoid,
\[
\xymatrix{
G=\quad G_2\ar[r]^-m&G_1\ar@(ru,lu)[]_i\ar@<1.2ex>[rr]^r\ar@<-1.2ex>[rr]_d&&G_0\ar[ll]|u
}\;,
\]
where $G_2$ is the pullback of the \emph{domain} and \emph{range} morphisms:
\[
\xymatrix{
G_2 \ar[r]^{\pi_1} \ar[d]_{\pi_2} & G_1 \ar[d]^{r}  \\
 G_1\ar[r]_{d}  & G_0
}
\]
\begin{example}

\begin{itemize}
\item A \emph{topological groupoid} is an internal groupoid in $\Top$.
\item A \emph{Lie groupoid} is an internal groupoid in the category of smooth manifolds
such that $d$ is a submersion (so that $G_2$ is well defined).
\item A \emph{localic groupoid} is an internal
groupoid in $\Loc$.
\end{itemize}
\end{example}
A localic groupoid $G$ is said to be open if $d$ is an open map. Thus, if $G$ is
open the multiplication map $m$ is also an open map. A \emph{localic \'etale groupoid} is
a localic groupoid such that $d$ is a local homeomorphism, in which case all
the structure maps are local homeomorphisms and, hence, $G_0$ is isomorphic
to an open sublocale of $G_1$.


\subsection{Groupoids and stably supported quantales}\label{unitalsupports}

Let $G$ be an open localic groupoid. Since the multiplication map $m$ is open, there is a sup-lattice homomorphism defined as the following composition (in $\SL$):
\[
\xymatrix{
\opens(G_1)\otimes \opens(G_1) \ar@{->>}[r] & \opens(G_2) \ar[r]^{m_!} &  \opens(G_1)\;.
}
\]
This defines an associative multiplication on $\opens(G_1)$ which together with the
isomorphism 
\[
\xymatrix{
\opens(G_1) \ar[r]^{i_!} & \opens(G_1)
}\;.
\]
makes $\opens(G_1)$ an involutive quantale. This quantale is denoted by $\opens(G)$. It is a worth mentioning that the involutive quantale $\opens(G)$ of an open groupoid $G$ is unital if and only
if $G$ is \'etale, in which case the unit is $e = u_!(1)$ and $u_!$ defines an order-isomorphism
\[u_! : \opens(G_0)\to \downsegment(e)=\{  a\in \opens(G): a\leq e  \}\;. \]
Hence, in particular, $\downsegment(e)$ is a frame.

Let $Q$ be a unital involutive quantale. We recall that by a \emph{support} on $Q$ is meant a sup-lattice homomorphism $\spp: Q\to Q$  satisfying the following conditions for all $a\in Q$:
\begin{align*}
\spp(a) &\leq e \\
\spp(a) &\leq aa^*\\
a &\leq \spp(a)a\;.
\end{align*}
The support is said to be \emph{stable} if in addition we have, for all $a, b\in Q$,
\begin{align*}
\spp(ab) &= \spp(a\spp(b))\;.
\end{align*}
We remark that the quantale $\opens(G)$ of an \'etale groupoid $G$ has a stable support given by $u_!\circ d_!:\opens(G)\to \opens(G)$.

For any unital quantale $Q$ with a support, the following equalities hold for all $a, b\in Q$:
\begin{align*}
\spp(a)1 & = a1\;,\\
\spp(b)   & = b \quad\text{if }b\leq e\;.
\end{align*}
The unital involutive subquantale $\downsegment(e)$ has trivial involution and it is a locale with multiplication equal to $\wedge$. We denote this locale by $Q_0$ and refer to it as the \emph{base locale} of $Q$.
We further recall that any stably supported quantale $Q$ admits a unique support, given by the following formulas,
\begin{align}
\spp(a) &= a1\wedge e\;,\\
\spp(a) &= aa^*\wedge e\;,
\end{align}
and, if $Q$ is a stably supported quantale, for all $b\in Q_0$ and $a\in Q$ we have
\begin{align*}
ba & = b1\wedge a\;.
\end{align*}
Moreover, a support is stable if and only if it is a homomorphism of $Q_0$-modules; that is, for all $b\in Q_0$ and $a\in Q$ we have
\begin{align*}
\spp(ba) & = b\spp(a)\;.
\end{align*}
A unital involutive quantale equipped with a stable support is said to be \emph{stably supported}. Having a stable support is a property rather than structure on a unital involutive quantale, and homomorphisms of unital involutive quantales between stably supported quantales automatically commute with the supports. Denoting the full subcategory of $\uquantales$ whose objects are the stably supported quantales by $\StabQu$, we have

\begin{theorem}\cite{Re07}
$\StabQu$ is a reflective full subcategory of $\uquantales$.
\end{theorem}

By a \emph{stable quantal frame} is meant a stably supported quantale which is also a locale. The following equivalent conditions hold for all stable quantal frames:
\begin{align*}
(a\wedge e)1 &\geq \V_{yz^*\leq a}y\wedge z\;,\\
(a\wedge e)1 &\geq \V_{xx^*\leq a}x\;.
\end{align*}

An \emph{inverse quantal frame} is a stable quantal frame $Q$ that satisfies 
\begin{align*}
\V \ipi(Q) &= 1,
\end{align*}
where $\ipi(Q)=\{  s\in Q: ss^* \vee s^*s \leq e  \}$ is the set of \emph{partial units} of $Q$.
The inverse quantal frames $Q$ are precisely the quantales of the form $Q\cong \opens(G)$ for a localic \'etale groupoid $G$. We also recall that the \emph{category of inverse quantal frames} $\iqf$ has the (necessarily involutive) homomorphisms of unital quantales as morphisms.


\section{Supports}\label{supports}\label{supportedquantales}


\subsection{Based quantales}

\begin{definition}\label{principal, def: AAbimodule}
Let $A$ be a locale. An \emph{$A$-$A$-bimodule} $M$ is a sup-lattice equipped with two unital (resp.\ left and right) $A$-module structures
\[ (a,m)\mapsto \lres{a}{m}\quad \text{and}\quad (a,m)\mapsto \rres{m}{a}\] 
satisfying the following additional condition for all $a, b \in A$
and $m \in M$:
\begin{align}\label{associativitybimodules}
  \rres{(\lres{a}{m})}{b} & = \lres{a}{(\rres{m}{b})}\;.  
\end{align}
A \emph{homomorphism} of $A$-$A$-bimodules is a sup-lattice homomorphism that preserves both actions.
\end{definition}

The notation for the left and the right action is meant to convey the idea of restriction, as in the following example:
\begin{example}
Let $A$ and $M$ be locales, and $d,c:M\to A$ two maps. The frame homomorphisms $d^*,c^*:A\to M$ make $M$ an $A$-$A$-bimodule with the actions defined by
\begin{align*}
\lres{a}{m} &= d^*(a)\wedge m\;,\\
\rres{m}{a} &= c^*(a)\wedge m\;.
\end{align*}
\end{example}

\begin{definition}\label{principal, def: AAquantale}
By an \emph{$A$-$A$-quantale} $Q$ is meant an $A$-$A$-bimodule equipped with a quantale
multiplication $(x, y)\mapsto xy$ satisfying the following additional conditions
for all $a \in A$ and $x, y \in Q$:
\begin{align}
 (\lres{a}{x})y & =  \lres{a}{(xy)}\;,\label{AAquantale1}\\
 (\rres{x}{a})y & =   x(\lres{a}{y})\;,\label{AAquantale2}\\
 \rres{(xy)}{a} & =   x(\rres{y}{a})\;.\label{AAquantale3}
\end{align}
\end{definition}

The second condition is equivalent to stating that the quantale multiplication on $Q$ is well defined as a sup-lattice homomorphism $Q\otimes_A Q\to Q$, and the two other conditions say that this is actually a homomorphism of $A$-$A$-bimodules. Hence, an $A$-$A$-quantale is just a semigroup in the monoidal category of $A$-$A$-bimodules.

\begin{lemma}\label{lem:restrictionhom}
Let $Q$ be an $A$-$A$-quantale. If $Q$ is unital then for all $a,b\in A$ we have
\[
\lres{(ab)}e=(\lres a e)(\lres b e)\;.
\]
Hence, the mapping $a\mapsto\lres a e$ is a homomorphism of unital quantales $A\to Q$.
\end{lemma}

\begin{proof}
Let $a,b\in A$. Then
$
\lres{(ab)}e = \lres a{(\lres b e)}=\lres a{\bigl( e(\lres b e)\bigr)}=(\lres a e)(\lres b e)
$. \qed
\end{proof}

\begin{definition}\label{principal, def: involutive}
An $A$-$A$-quantale is \emph{involutive} if it is an involutive semigroup; the involution is denoted by $a\mapsto a^{*}$ and it is required to satisfy, besides the standard conditions $x^{**}=x$ and $(xy)^{*}=y^{*}x^{*}$, the following two conditions:
\begin{align}
 (\V_i x_i)^{*} & =\V_i x^{*}_i\;, \\
 (\lres{a}{\rres{x}{b}})^{*} & = \lres{b}{\rres{x^{*}}{a}}\;.
\end{align}
\end{definition}

\begin{remark}
An involutive $A$-$A$-quantale is not the same as an involutive semigroup in the category of $A$-$A$-bimodules. The latter would require $(-)^*$ to be a homomorphism of bimodules, hence satisfying $(\lres{a}{\rres{x}{b}})^{*}= \lres{a}{\rres{x^{*}}{b}}$\;.
\end{remark}

\begin{definition}
By a \emph{based quantale} will be meant an involutive quantale $Q$ equipped with the structure of an involutive $Q_0$-$Q_0$-quantale for some locale $Q_0$. A \emph{homomorphism} of based quantales $f:Q\to R$ consists of a pair $(f_1,f_0)$ where $f_1:Q\to R$ is a homomorphism of involutive quantales and $f_0:Q_0\to R_0$ is a homomorphism of locales such that the following conditions hold for all $x\in Q$ and $a\in Q_0$:
\begin{align*}
f_1\bigl(\lres a x\bigr) &=\lres{f_0(a)}{f_1(x)}\;;\\
f_1\bigl(\rres x a\bigr) &=\rres{f_1(x)}{f_0(a)}\;.
\end{align*}
The category thus obtained is called the \emph{category of based quantales} and it is denoted by $\basedQu$. By a \emph{strong} homomorphism of based quantales will be meant a homomorphism $f$ such that $f_1$ is a strong homomorphism of involutive quantales. The subcategory of $\basedQu$ whose homomorphisms are strong is denoted by $\stbasedQu$.
\end{definition}


\subsection{General supports}

\begin{definition}\label{principal, def: nonunitalsupp}
A based quantale $Q$ is \emph{supported} if it is equipped with a sup-lattice homomorphism $\spp_Q: Q\rightarrow Q_0$ (denoted simply $\spp$ if no confusion will arise), called the \emph{support}, which satisfies the following conditions for all $x,y\in Q$:
\begin{align}
\spp(1_Q)&=1_{Q_0}\;,\label{spp1} \\ 
\lres {\spp(x)}{y} &\leq xx^{*}y\;, \label{spp2}\\
\lres {\spp(x)}{x} &= x \label{spp3}\;.
\end{align}
By a \emph{homomorphism} $f:Q\to R$ of supported quantales is meant a homomorphism of based quantales $(f_1,f_0)$ that commutes with the supports; that is, such that for all $x\in Q$ we have
\[
f_0(\spp_Q(x)) = \spp_R(f_1(x))\;.
\]
The category thus obtained is called the \emph{category of supported quantales} and it is denoted by $\sppQu$. We shall denote by $\stsppQu$ the subcategory of $\sppQu$ whose homomorphisms are strong, and by $\sppuQu$ the subcategory of $\sppQu$ whose objects are unital quantales and whose homomorphisms are unital.
\end{definition}

\begin{example}\label{exm:classicalspp}
Let $Q$ be a unital involutive quantale with a support $\spp$, and let $Q_0=\downsegment(e)$. Then $Q$, with the $Q_0$-$Q_0$-quantale structure given by change of base along the inclusion $Q_0\to Q$, is supported in the sense of Definition~\ref{principal, def: nonunitalsupp}.
\end{example}

For a unital support $\spp$ with $Q_0=\downsegment(e)$ as in the above example it is also true that $\spp:Q\to Q_0$ is surjective, but for non-unital quantales this is not the case in general, as the following example shows.

\begin{example}\label{stablenotequi}
Let $Q$ and $A$ be locales and let $s:Q\to A$ and $r:A\to Q$ be homomorphisms of locales such that $r\circ s=\ident_Q$ (so $Q$ is a retract of $A$ in $\Frm$). Then $Q$, regarded as a commutative quantale with trivial involution, together with the (both left and right) action of $A$ on $Q$ which is defined by change of base along $r$, yields a supported quantale with $Q_0=A$ and support $\spp=s$.
\end{example}

\begin{lemma}\label{propspp}
Let $Q$ be a supported quantale. The following conditions hold for all $x,y\in Q$:
\begin{enumerate}
 \item\label{propspp1} $(\lres {\spp(x)}{y})^*=(\rres {y^*}{\spp(x)} )$\;;
 \item\label{propspp2} $\rres{y}{\spp(x)}\leq yxx^{*}$\;;
 \item\label{propspp3} $x\leq xx^{*}x$\;;
 \item\label{propspp4} $x\leq x1_Q$\;;
 \item\label{propspp5} $1_Q1_Q=1_Q$\;;
 \item\label{propspp6} $\lres{\spp(x1_Q)}{1_Q}=x1_Q$\;;
 \item\label{propspp7} The sup-lattice homomorphism $\spp(Q)\to \rs(Q)$ defined by $x\mapsto \lres{x}{1_Q}$ is a retraction split by the section $\rs(Q)\to \spp(Q)$ which is defined by $x\mapsto \spp(x)$.
\end{enumerate}
\end{lemma}

\begin{proof}
\eqref{propspp1} follows from $(\lres {\spp(x)}{y})^* = (\lres{\spp(x)}{\rres{y}{1_A}})^* 
   = \lres{1_A}{\rres{y^*}{\spp(x)}}
   = \rres{y^*}{\spp(x)}$.

\eqref{propspp2} follows from \eqref{propspp1} and the properties of the involution: $\rres{y}{\spp(x)}=\rres{y^{**}}{\spp(x)}=(\lres{\spp(x)}{y^*})^{*} \leq (xx^{*}y^{*})^*=yxx^{*}$.

\eqref{propspp3} is an immediate consequence of the axioms: $x=\lres{\spp(x)}{x}\leq xx^{*}x$.

\eqref{propspp4} and \eqref{propspp5} follow immediately from \eqref{propspp3}.

\eqref{propspp6} follows from \eqref{propspp5} and the properties of the involution:
\begin{align*}
\lres{\spp(x1_Q)}{1_Q}&\leq (x1_Q)(x1_Q)^{*}1_Q=x1_Q x^{*}1_Q\leq x1_Q\\
&=\lres{\spp(x1_Q)}{x1_Q}\le\lres{\spp(x1_Q)}{1_Q}\;.
\end{align*}

In order to verify \eqref{propspp7} let $x\in \rs(Q)$. So $x=x1_Q$ and, using \eqref{propspp6}, we obtain $\lres{\spp(x)}{1_Q}=\lres{\spp(x1_Q)}{1_Q}=x1_Q=x$. \qed
 
\end{proof}

We remark that supported quantales are in fact strongly Gelfand (\ie\ they satisfy $a\leq aa^*a$ for all $a\in Q$), which implies that they are stably Gelfand and in turn implies that they are Gelfand (see, \cite{MP1}). The latter implies that $T(Q)$ (the sub-quantale of two sided-elements of $Q$) consists of self-adjoint elements for strongly Gelfand quantales: if $a\in T(Q)$ then
\[ a^*\leq a^*aa^* \leq 1_Q a 1_Q \leq a\;, \]
so $a^* ² a$\;, and thus $a=a^*$.


\subsection{Unital supports}

Let us trim our terminology so as to better relate the notion of support just introduced to the original one of \cite{Re07}:

\begin{definition}
If $Q$ is a unital quantale, by a \emph{unital support} on $Q$ will be meant a support in the sense of section~\ref{unitalsupports}; that is, a sup-lattice homomorphism $\spp_Q:Q\to Q$ (or simply $\spp$) satisfying the following three conditions for all $x\in Q$:
\begin{align*}
\spp(x)&\le e\;;\\
\spp(x)&\le xx^*\;;\\
x&\le \spp(x)x\;.
\end{align*}
We shall call such a quantale \emph{unitally supported}. By a \emph{homomorphism} of unitally supported quantales $f:Q\to R$ will be meant a homomorphism of unital involutive quantales that commutes with the unital supports; that is, such that for all $x\in Q$ we have
\[
f(\spp_Q(x))=\spp_R(f(x))\;.
\]
The category thus obtained will be denoted by $\usppQu$. (Not to be confused with $\sppuQu$, \cf\ Definition~\ref{principal, def: nonunitalsupp}.)
\end{definition}

Given any homomorphism $f$ in $\usppQu$ there is a homomorphism $(f,f\vert_{\downsegment(e)})$ in $\sppuQu$. This yields a faithful functor
$U:\usppQu\to\sppuQu$ that is also injective on objects (\cf\ Example~\ref{exm:classicalspp}), so we shall identify $\usppQu$ with a subcategory of $\sppuQu$.

\begin{lemma}\label{fromspptouspp}
$\usppQu$ is a full and reflective subcategory of $\sppuQu$. To each object $(Q,\spp)$ of $\sppuQu$ the reflection $\eta:(Q,\spp)\to(Q,\spp_e)$ is obtained by defining the unital support $\spp_e$ for all $x\in Q$ by
\[
\spp_e(x)=\lres{\spp(x)}e\;,
\]
and setting $\eta_1=\ident_Q$ and $\eta_0(a)=\lres a e$ for all $a\in Q_0$.
\end{lemma}

\begin{proof}
Proving that $\spp_e$ is a unital support is straightforward. And $\eta$ is a morphism in $\sppuQu$ because $\eta_0$ is a homomorphism of locales (\cf\ Lemma~\ref{lem:restrictionhom}), and it commutes with the supports as required:
\[
\eta_0(\spp(x)) = \lres{\spp(x)}e=\spp_e(x)=\spp_e(\eta_1(x))\;.
\]
Let $R$ be another unitally supported quantale, with support $\spp_R$, and let $f:Q\to R$ be a homomorphism in $\sppuQu$. Let us verify that $f_1:Q\to R$ is a homomorphism in $\usppQu$:
\[
f_1(\spp_e(x))=f_1(\lres{\spp(x)}e)=\lres{f_0(\spp(x))}{f_1(e)}=\lres{\spp_R(f_1(x))}{e}=\spp_R(f_1(x))\;.
\]
(The rightmost equality is a consequence of the identification $\lres a x=ax$ for objects of $\usppQu$.)
In order to see that $f_1:Q\to R$ makes the following diagram in $\sppuQu$ commute
\[
\xymatrix{
Q\ar[rr]^\eta\ar[drr]_f&&Q\ar[d]^{f_1}\\
&&R
}
\]
we need to verify both
\begin{align*}
f_1\circ\eta_0&=f_0\;,\\
f_1\circ\eta_1&=f_1\;.
\end{align*}
For the first equation we observe that for all $a\in Q_0$ we have
\[
f_1\circ\eta_0(a)=f_1(\lres a e)=\lres{f_0(a)}e=f_0(a)\;.
\]
The second equation is immediate because $\eta_1=\ident_Q$. This also implies that $f_1$ is the unique homomorphism in $\usppQu$ that makes the diagram above commute, hence proving that the subcategory is reflective. In order to see that it is also a full subcategory it suffices to notice that the reflection is idempotent in the sense that for $Q\in\usppQu$ the map $\eta:Q\to Q$ is an isomorphism (so the adjunction is a reflection). \qed
\end{proof}


\subsection{Stable supports}

\begin{definition}
Let $Q$ be a supported quantale. The support $\spp$ is \emph{stable} if $\spp{(xy)}\leq \spp(x)$ for all $x,y\in Q$. In this case $Q$ is called \emph{stably supported}.
\end{definition}

\begin{lemma}\label{stable}
Let $Q$ be a supported quantale. The following properties are equivalent:

\begin{enumerate}
\item $\spp(xy)\leq \spp(x)$ for all $x,y\in Q$\quad (the support is stable)\;;
\item $\spp(x1_Q)=\spp(x)$ for all $x\in Q$\;;
\item $\spp{(xy)}=\spp(\rres{x}{\spp(y)})$ for all $x,y\in Q$\;.
\end{enumerate}
\end{lemma}

\begin{proof}
$(1\Rightarrow 2)$ Since $x\leq x1_Q$, from stability we obtain   
$\spp(x)\leq\spp(x1_Q)\leq\spp(x)$.

$(2\Rightarrow 1)$ This follows from $\spp(xy)\leq \spp(x1_Q)=\spp(x)$.

$(3\Rightarrow 1)$ This is immediate because $\rres{x}{\spp(y)}\leq x$.

$(1\Rightarrow 3)$ Assuming that $\spp$ is stable, and applying Lemma~\ref{propspp}\eqref{propspp2}, we have:
\[
\spp(xy)  =  \spp(x(\lres{\spp(y)}{y})) = \spp((\rres{x}{\spp(y)})y) \leq \spp(\rres{x}{\spp(y)}) \leq \spp(xyy^*)\leq \spp(xy)\;. \qed
\]
\end{proof}

Note that, contrary with the situation for unital stably supported quantales, we do not necessarily have $\rs(Q)\cong Q_0$, as the above Example~\ref{stablenotequi} shows. But, if $\spp$ is stable and $Q$ has a multiplication unit $e$, the unital support $\spp_e$ of Lemma~\ref{fromspptouspp} is obviously stable and thus $\rs(Q)\cong\downsegment(e)$.


\subsection{Equivariant supports}

\begin{definition}\label{defequivariant}
Let $Q$ be a based quantale. A support $\spp$ is said to be \emph{equivariant} if for all $a\in Q_0$ and $x\in Q$ we have
\begin{equation}
\spp(\lres{a}{x})=a\wedge \spp(x)\;.
\end{equation}
A supported quantale is called \emph{equivariantly supported} whenever the support is equivariant. The full subcategory of $\sppQu$ whose objects are the equivariantly supported quantales is denoted by $\eqsppQu$. The subcategory of $\eqsppQu$ whose homomorphisms are strong is denoted by $\steqsppQu$, and the subcategory of $\eqsppQu$ whose objects are unital quantales and whose homomorphisms are unital is denoted by $\ueqsppQu$.
\end{definition}

\begin{lemma}\label{adjoint}
Let $Q$ be a based quantale, and let $\spp$ be an equivariant support. Then $\spp:Q\to Q_0$ is left adjoint to the sup-lattice homomorphism $\lres{(.)}{1_Q}:Q_0\rightarrow Q$, and, moreover, the adjunction is a reflection. Hence, the support $\spp$ is uniquely determined and $\lres{(.)}{1_Q}$ also preserves arbitrary meets.
\end{lemma}

\begin{proof}
The unit of the adjunction is
\[
x\le\lres {\spp(x)}{1_Q}
\]
and it follows immediately from the axiom $\lres{\spp(x)}{x}=x$ of supports. 
The counit is the condition 
\[
\spp(\lres{a}{1_Q})\leq a\;,
\]
which is proved using equivariance: 
\[
\spp(\lres{a}{1_Q}) =a\wedge \spp(1_Q) = a \wedge 1_{Q_0}  = a\;.
\]
Moreover, this also shows that the counit is an equality, so $\spp$ is left adjoint to $\lres{(.)}{1_Q}$ and the adjunction is a reflection. \qed
\end{proof}

\begin{corollary}\label{principal, cor: rq}
Let $Q$ be an equivariantly supported quantale. Then the map $Q_0\rightarrow \rs(Q)$ defined by $x\mapsto \lres{x}{1_Q}$ is an order isomorphism whose inverse is the map $\spp:\rs(Q) \rightarrow Q_0$. Hence, in particular, $\rs(Q)$ is a locale. 
\end{corollary}

\begin{proof}
By Lemma~\ref{adjoint} we have $\spp(\lres{a}{1_Q}) = a$
for all $a\in Q_0$. And, due to Lemma~\ref{propspp}\eqref{propspp7}, we have $\lres{\spp(x)}{1_Q}=x$ for all $x\in \rs(Q)$. \qed
\end{proof}

\begin{lemma}\label{equi-stab}
Every equivariant support is stable.
\end{lemma}
\begin{proof}
Let us assume that $\spp$ is an equivariant support and show that $\spp{(xy)}\leq \spp(x)$ for all $x,y\in Q$. Indeed,
\[
\spp(xy) =\spp((\lres{\spp(x)}{x})y) = \spp(\lres{\spp(x)}{(xy)})  = \spp(x)\wedge \spp(xy)  \leq  \spp(x)\;. \qed
\]
\end{proof}

\begin{remark}
We would like to remark that notion of equivariantly supported quantale generalizes the notion an Ehresmann semigroup  (see \cite{KL}*{section 2}) as follows: Let $Q$ be an equivariantly supported $A$-$A$-quantale, then by considering the actions of $A$ as a generalization of the multiplication of elements of $A$ by elements of $Q$, we can define mappings $\lambda=\spp:Q\to A$ and $\rho=\spp\circ(-)^*:Q\to A$ such that
\begin{itemize}
\item Due to Corollary~\ref{principal, cor: rq} $A$ can be identified with $R(Q)$, so for all $a\in R(Q)$, we have $\lambda(a)=a$ and $\rho(a)=a$ .
\item Clearly $a\lambda(a)=a$ and $\rho(a)a=a$ hold for any $a\in Q$ due to Definition~\ref{principal, def: nonunitalsupp}.
\item $\lambda(\lambda(a)b)=\lambda(ab)$ and $\rho(a\rho(b))=\rho(ab)$ hold for any $a,b\in Q$ due to Lemma~\ref{stable}.
\end{itemize} 
Finally, we note that the action of $R(Q)$ on $Q$ does not coincide with multiplication in $Q$, furthermore $R(Q)$ is not a commutative subsemigroup.
\end{remark}

\begin{remark}
If $Q$ is a unitally supported quantale, the notions of equivariance and stability coincide \cite{GSQS}*{Th.\ 2.20}, but for general supports this is not the case, as the stable support of Example~\ref{stablenotequi} shows.
\end{remark}

\begin{remark}
Note that if $Q$ is equivariantly supported then $\spp\bigl(\lres{(\cdot)} x\bigr):Q_0\to Q_0$ preserves arbitrary non-empty meets, for if $S\subset Q_0$ is non-empty then 
\[
\spp\bigl(\lres{\bigwedge S}{x}\bigr)=\bigwedge S\wedge \spp(x)=\bigwedge_{s\in S}(s\wedge \spp(x))=\bigwedge_{s\in S}\spp(\lres{s}{x})\;.
\]
\end{remark}

Lemma~\ref{adjoint} shows that having an equivariant support is equivalent to the statement that the mapping $\lres{(\cdot)}1$ have a left adjoint satisfying the axioms of a support, and thus being equivariantly supported can be regarded as a property rather than structure on a based quantale. Hence, we can see $\eqsppQu$ as being a subcategory of $\basedQu$, and we have the following fact.

\begin{theorem}
$\eqsppQu$ is a reflective subcategory of $\basedQu$.
\end{theorem}

\begin{proof}
This follows from ``universal sup-lattice algebra'' if we regard the theory of based quantales and of equivariantly supported quantales as being two-sorted theories of \emph{sup-algebras}, by which are meant sup-lattices equipped with sup-lattice multi-morphisms subject to equational laws. Free sup-algebras exist and can be presented by generators and relations by means of the same techniques used for locales, quantales and modules, such as nuclei for describing quotients. The adaptation of these techniques in order to ``freely'' generate equivariant supports for based quantales is straightforward. \qed
\end{proof}

However, iterating the construction of an equivariantly supported quantale from a based quantale does not stabilize in general because the adjunction of which the inclusion functor is a right adjoint is not necessarily a reflection (equivalently, the inclusion functor is not full). For the latter to hold additional conditions are necessary, as we now describe.

\begin{lemma}\label{condiionssp}
Let  $Q$ be an equivariantly supported quantale. Let $x\in Q$ and $a\in Q_0$, and suppose that the following two conditions hold:
\begin{enumerate}
\item $\lres{a}{x}=x$\;;
\item $\lres{a}{1_Q}\leq x1_Q$\;.
\end{enumerate} 
Then $a=\spp(x)$\;.
\end{lemma}

\begin{proof}
By applying equivariance we conclude that
\[
a\wedge \spp(x)=\spp(\lres a x)=\spp(x)\;,
\]
so $\spp(x) \leq a$. For the converse inequality, again equivariance gives us 
\[
a=a\wedge 1_{Q_0}=a\wedge \spp(1_Q)=\spp(\lres{a}{1_Q})\leq \spp(x1_Q)\leq \spp(x)\;. \qed
\]  
\end{proof}

\begin{lemma}\label{fullsubcategoryeqb}
Let $Q$ be a supported quantale and $R$ an equivariantly supported quantale. Then any strong homomorphism of based quantales $f:Q\to R$ commutes with the supports.
\end{lemma}

\begin{proof}
Let $f:Q\rightarrow R$ be a homomorphism of based quantales. Then for all $x,y\in Q$ we have
\begin{align*}
\lres{f_0(\spp_Q(x))}{f_1(x)}&=f_1(\lres{\spp_Q(x)}{x})=f_1(x)\;,\\
\lres{f_0(\spp_Q(x))}{f_1(y)}&=f_1(\lres{\spp_Q(x)}{y})\leq f_1(xx^{*}y)\leq f_1(x)1_R\;.
\end{align*}
In particular, if $f$ is a strong homomorphism then
\[
\lres{f_0(\spp_Q(x))}{1_R}=\lres{f_0(\spp_Q(x))}{f_1(1_Q)}\leq f_1(x) 1_R\;,
\]
and thus, by Lemma~\ref{condiionssp}, we conclude that $f_0(\spp_Q(x))=\spp_R(f_1(x))$. \qed
\end{proof}

The hypothesis of strong homomorphisms in Lemma~\ref{fullsubcategoryeqb} is necessary, as the following example shows.

\begin{example}
Let $Q=\mathbf 2$ be the two element locale $\{0,1\}$ with $0<1$, regarded as  an equivariantly (and unitary) supported quantale with $Q_0=Q$ and support $\spp=\ident_Q:Q\to Q_0$. Let the zero endomorphism $\mathbf 0:Q\to Q$ be defined by $\mathbf 0_1(1)=0$ and $\mathbf 0_0=\ident$. This is a homomorphism of based quantales, but it does not commute with the supports:
\[
\spp(\mathbf 0_1(1))=\spp(0)= 0\neq 1 =\mathbf 0_0(1)=\mathbf 0_0(\spp(1))\;.
\]
\end{example}

So if we restrict to strong homomorphisms the universal construction of an equivariant support on a based quantale provides us with a well defined reflection which is stable under iteration:

\begin{corollary}
$\steqsppQu$ is a reflective full subcategory of $\stbasedQu$.
\end{corollary}

At a first sight this appears to be a weaker property than that of unitary supported quantales \cite{Re07}*{Th.\ 3.10}, where homomorphisms are not required to be strong. However, we note that the latter property does not apply to arbitrary homomorphisms of involutive quantales, either, since it holds for unital quantales and unital homomorphisms.

We also note the following fact concerning unital quantales.

\begin{corollary}\label{cor:backtostabqu}
$\ueqsppQu$ is isomorphic to the category of stably (unitally) supported quantales $\ssq$ of \cite{Re07}.
\end{corollary}


\section{Quantal frames}\label{principal, section: quantalframes}


\subsection{Supported quantal frames}

\begin{definition}\label{principal, def: basedquantalframe}
By a \emph{based quantal frame} is meant a based quantale $Q$ such that for all $x,y,y_i\in Q$ and $a\in Q_0$ the following properties hold:
\begin{align}
x \wedge \V_i y_i &= \V_i x\wedge y_i\;,\\
(\lres{a}{x})\wedge y &= \lres{a}{(x\wedge y)}\;,\label{qf1}\\
(\rres{x}{a})\wedge y &= \rres{(x\wedge y)}{a}\;. \label{qf2}
\end{align}
By a \emph{supported quantal frame}, \emph{stably suported quantal frame}, and \emph{equivariantly supported quantal frame}, is meant a based quantal frame equipped with a support, a stable support, and an equivariant support, respectively.
\end{definition}

\begin{example}\label{exm:ustabqu}
Any stably supported quantale $Q$ in the unital sense of \cite{Re07} satisfies \eqref{qf1}, since for $x\in Q$ and $a\in\downsegment(e)$ we have $a1\wedge x=ax$, and thus
\[
\lres a{(x\wedge y)}=a(x\wedge y)=a1\wedge x\wedge y=ax\wedge y=\bigl(\lres a x\bigr)\wedge y\;.
\]
Similarly, it satisfies \eqref{qf2}. Hence, any stable quantal frame in the sense of \cite{Re07} is a based quantal frame.
\end{example}

Recall from Corollary~\ref{principal, cor: rq} that for every equivariantly supported quantale $Q$ the sup-lattice of right sided elements $\rs(Q)$ is order-isomorphic to $Q_0$ and therefore it is a locale. In particular, any $x\in \rs(Q)$ has the following representation:
\begin{equation}\label{eq: representationrestriction}
\lres{\spp(x)}{1_Q}=x\;.
\end{equation} 

\begin{lemma}\label{principal, cor: meet}
Let $Q$ be an equivariantly supported quantal frame. Then, for all $x\in \rs(Q)$ and $y\in Q$, we have:
\begin{equation}
\spp(x\wedge y) = \spp(x)\wedge \spp(y)\;.
\end{equation}

\end{lemma}
\begin{proof}
\begin{align*}
\spp(x\wedge y) & = \spp\bigl((\lres{\spp(x)}{1_Q})\wedge y\bigr)\\
& = \spp\bigl(\lres{\spp(x)}{(1_Q\wedge y)}\bigr)\\
& = \spp\bigl(\lres{\spp(x)}{y}\bigr)\\
& = \spp(x)\wedge \spp(y)\;. \qed
\end{align*}
\end{proof}

As a consequence of Corollary~\ref{cor:backtostabqu} and Example~\ref{exm:ustabqu}, if we add a multiplicative unit we recover the theory of stable quantal frames:

\begin{corollary}\label{stably1}
The class of unital equivariantly supported quantal frames coincides with the class of stable quantal frames of \cite{Re07}.
\end{corollary}


\subsection{Principal quantales}

We adopt the following terminology which is motivated by the relation to principal bundles on groupoids (\cf\ section \ref{subsec: principal groupoids}).

\begin{definition}\label{principal, def: effective}
By a \emph{principal quantale} $Q$ will be meant an equivariantly supported quantal frame $Q$ satisfying the equivalent conditions of the following Theorem~\ref{thm:effective}.
\end{definition}

\begin{theorem}\label{thm:effective}
Let $Q$ be an equivariantly supported quantal frame. The following conditions are equivalent:
\begin{enumerate}
\item\label{thm:effective1} $Q\cong \rs(Q)\otimes_{\ts(Q)} \ls(Q)$ in $\Frm$.
\item\label{thm:effective2} For each locale $S$ and locale homomorphisms $h: \rs(Q)\to S$ and $k:\ls(Q)\to S$ such that $\rres{h}{\ts(Q)}=\rres{k}{\ts(Q)}$, there is a unique locale homomorphism $t:Q\to S$ such that $\rres{t}{\rs(Q)}=h$ and $\rres{t}{\ls(Q)}=k$.
\item\label{thm:effective3} The triple $(Q,\lres{(.)}{1_Q},\rres{1_Q}{(.)})$ is the cokernel pair in $\Frm$ of the inclusion
\[i: E \to A\;,\]
where $E=\{ a\in A\mid \lres{a}{1_Q}=\rres{1_Q}{a}  \}$.
\end{enumerate}
\end{theorem}

\begin{proof}
The equivalence $\eqref{thm:effective1} \Leftrightarrow \eqref{thm:effective2}$ follows from the universal property of the pushout in $\Frm$, taking into account that $\rs(Q)\otimes_{\ts(Q)} \ls(Q)$ is the pushout of the square
\[
\xymatrix{
\ts(Q)\ar@{^{(}->}[r] \ar@{^{(}->}[d]  &\ls(Q)\ar@{^{(}->}[d]\\
\rs(Q)\ar@{^{(}->}[r]& \rs(Q)\otimes_{\ts(Q)} \ls(Q)
}\;.
\]
In order to prove $\eqref{thm:effective2}\Rightarrow \eqref{thm:effective3}$ let $S$ be any frame and let $h':A\to S$ and $k':A\to S$ be frame homomorphisms such that $h'(b)=k'(b)$ for any $b\in E$, as the following diagram indicates:
\[
\xymatrix{
Q && A  \ar@<0.5ex>[d]^{k'}\ar@<-0.5ex>[d]_{h'}   \ar@<0.5ex>[ll]^{\rres{1_Q}{(.)}}\ar@<-0.5ex>[ll]_{\lres{(.)}{1_Q}} & E \ar@{_{(}->}[l]_i \\
&& S 
}
\]
We need to show that there is a unique frame homomorphism $t:Q\to S$ such that $t\circ \lres{(.)}{1_Q}=h'$ and $t\circ \rres{1_Q}{(.)}=k'$. Since $(Q,\spp)$ is an equivariantly supported $A$-$A$-quantale, we know that $\rs(Q)\cong A$ and $\ls(Q)\cong A$ (\cf\  Corollary~\ref{principal, cor: rq}). This implies that both $\rs(Q)$ and $\ls(Q)$ are locales and so is $\ts(Q)$. Therefore, the maps $h=h'\circ \spp(.)$ and $k=k'\circ \spp((.)^*)$ are frame homomorphisms (\cf\ Lemma~\ref{principal, cor: meet}) from $\rs(Q)$ and $\ls(Q)$ to $S$, respectively. Moreover, due to \eqref{eq: representationrestriction} we have $\spp(b)\in E$ for all $b\in \ts(Q)$, since the two-sided elements of $Q$ are self-adjoint:
\begin{equation*}
\lres{\spp(b)}{1_Q}=b=\rres{1_Q}{\spp(b^*)}=\rres{1_Q}{\spp(b)}\;.
\end{equation*}
Hence, we obtain
$h(b)=h'(\spp(b))=k'(\spp(b))=k'(\spp(b^*))=k(b)$,
showing that $\rres{h}{\ts(Q)}=\rres{k}{\ts(Q)}$,
so by hypothesis there exists a unique  frame homomorphism $t:Q\to S$ such that $h=\rres{t}{\rs(Q)}$ and $k=\rres{t}{\ls(Q)}$. Then for all $a\in A$ we obtain
\[ t(\lres{a}{1_Q})=h(\lres{a}{1_Q})=h'(\spp(\lres{a}{1_Q}))=h'(a)\;,    \]
and, similarly, $t(\rres{1_Q}{a})=k'(a)$. This proves \eqref{thm:effective3}. Finally, let us prove $\eqref{thm:effective3}\Rightarrow \eqref{thm:effective2}$. Let $S$ be any frame and let $h:\rs(Q)\to S$ and $k:\ls(Q)\to S$ be two frame homomorphisms such that $\rres{h}{\ts(Q)}=\rres{k}{\ts(Q)}$. Consider the frame homomorphisms $h'=h(\lres{(.)}{1_Q})$ and $k'=k(\rres{1_Q}{(.)})$ from $A$ to $S$. For any $b\in E$ we have
\[   h'(b)=h(\lres{b}{1_Q})=h(\rres{1_Q}{b})=k(\rres{1_Q}{b})=k'(b)\;,   \] 
and by hypothesis the triple $(Q,\lres{(.)}{1_Q},\rres{1_Q}{(.)})$ is the cokernel pair of the frame inclusion $i: E \to A$. Then, there is a unique frame homomorphism $t:Q\to S$ such that $t(\lres{a}{1_Q})=h'(a)$ and $t(\rres{1_Q}{a})=k'(a)$. All we need to prove now is that $\rres{t}{\rs(Q)}=h$ and $\rres{t}{\ls(Q)}=k$. Let $a\in \rs(Q)$. Then
\[ t(a)=t(\lres{\spp(a)}{1_Q})=h'(\spp(a))=h(\lres{\spp(a)}{1_Q})=h(a)\;,  \]
and the other equality is proved similarly. This proves \eqref{thm:effective2}. \qed
\end{proof}

\begin{remark}
Let $(Q,\spp)$ be an equivariantly supported $A$-$A$-quantal frame. The frame homomorphism $f:\rs(Q)\otimes_{\ts(Q)} \ls(Q)\to Q$ given by $f(\V_i r_i\otimes l_i)=\V_i r_i\wedge l_i$ is injective if and only if $Q$ is a principal quantale. It will be clear from the results towards the end of the paper that not every equivariantly supported quantal frame is principal, but it is interesting to point out that the restriction of $f$ to the basis of pure tensors is always injective (and thus the basis is not downwards closed, due to \cite{Re07}*{Prop.\ 2.2}). Indeed, let us suppose that $r\wedge l = r'\wedge l'$. Then
\begin{align*}
r\otimes l & = (r\wedge 1_Qr)\otimes(l\wedge l1_Q)\\
	& = (r\wedge l1_Q)\otimes(l\wedge 1_Qr)\\
	& = (r1_Q\wedge l1_Q)\otimes(1_Qr\wedge 1_Ql)\\
	& = ((\lres{\spp(r)}{1_Q})\wedge l1_Q)\otimes(1_Qr \wedge (\rres{1_Q}{\spp(l^*)}))\\
	& = (\lres{\spp(r)}{(1_Q\wedge l1_Q)})\otimes(\rres{(1_Qr\wedge 1_Q)}{\spp(l^*)})\\
	& = (\lres{\spp(r)}{l1_Q})\otimes(\rres{1_Qr}{\spp(l^*)})\\
	& = (\lres{\spp(r)}{l})1_Q\otimes 1_Q(\rres{r}{\spp(l^*)})\\
	& = (\lres{\spp(r)}{(1_Q\wedge l)})1_Q\otimes1_Q(\rres{(r\wedge 1_Q)}{\spp(l^*)})\\
	& = ((\lres{\spp(r)}{1_Q})\wedge l)1_Q\otimes1_Q((r\wedge (\rres{1_Q}{\spp(l^*)}))& \text{[by \eqref{qf1}\ and\ \eqref{qf2}]}\\
	& = (r1_Q\wedge l)1_Q\otimes 1_Q(r\wedge 1_Ql)\\
	& = (r\wedge l)1_Q\otimes 1_Q(r\wedge l)\\
	& = (r'\wedge l')1_Q\otimes 1_Q(r'\wedge l')\\
        & = r'\otimes l'\;.
\end{align*}
\end{remark}

\begin{corollary}
Let $Q$ be a principal quantale. Then any element $q\in Q$ can be written as
\begin{equation}
q=\V\{ r\wedge l\mid r\in \rs(Q),\, l\in \ls(Q),\, r\wedge l\leq q \}\;.
\end{equation}
\end{corollary}

\begin{example}
Let $X$ be a set. The set of binary relations $\Rel(X)=\wp(X\times X)$ is a unital involutive quantale. The multiplication is the composition in the forward direction, $RS = S\circ R$, and the involution is reversal. Then $\Rel(X)$ is a principal quantale, because it is an equivariantly supported  $\wp(X)$-$\wp(X)$-quantale frame with
\[ \lres{U}{R}=(U\times X)\cap R\;,\quad \rres{R}{U}=(X\times U)\cap R\;,  \]
and equivariant support given by
\[ \spp(R)=\{(x,x)\mid (x,y)\in R\ \text{for some $y\in X$}\}\;. \]
Here we have that $\rs(\Rel(X))=\ls(\Rel(X))\cong \wp(X)$\;, and $\ts(\Rel(X))=\{ \emptyset, X\}$\;. Hence,
\[  \Rel(X)\cong \wp(X)\otimes \wp(X)\;.  \]
\end{example}


\subsection{Reflexive quantal frames}

\begin{definition}\label{principal, def: reflexive}
By a \emph{reflexive quantal frame} $(Q,\upsilon)$ will be meant an $A$-$A$-quantal frame equipped with a frame homomorphism $\upsilon: Q\to A$ that for all $a\in A$ satisfies
\begin{equation}
\upsilon(\lres{a}{1_Q})=a=\upsilon(\rres{1_Q}{a})\;.
\end{equation}
\end{definition}

\begin{lemma}\label{upsilon}
Let $(Q,\spp,\upsilon)$ be a stably (not necessarily equivariantly) supported $A$-$A$-quantale which is also a reflexive quantal frame. Then 
\begin{equation}
\upsilon(a1_Q)=\spp(a)
\end{equation}
for all $a\in Q.$
\end{lemma}

\begin{proof}
Let $a\in Q$. Then
\begin{align*}
\upsilon(a1_Q) & = \upsilon(\lres{\spp(a1_Q)}{1_Q})& \text{[by Lemma~\ref{propspp}\eqref{propspp6}]}\\
		       & = \upsilon(\lres{\spp(a)}{1_Q})& \text{($\spp$ is stable)}\\
    		       & = \spp(a)\;.\qed& \text{($\upsilon$ is reflexive)}		
\end{align*}
\end{proof}

\begin{lemma}\label{upsilon1}
Let $(Q,\spp,\upsilon)$ be an equivariantly supported $A$-$A$-quantale which is also a reflexive quantal frame. Then, for all $a\in Q$ and $b\in A$, we have 
\begin{equation}
\upsilon(\lres{b}{(a1_Q)})=b\wedge \upsilon(a1_Q)\;.
\end{equation}

\end{lemma}
\begin{proof}
Let $a\in Q$ and $b\in A$. Then
\begin{align*}
\upsilon(\lres{b}{(a1_Q)}) & = \upsilon((\lres{b}{a})1_Q)\\
 & = \spp(\lres{b}{a})& \text{(by Lemma~\ref{upsilon})}\\
 & = b \wedge \spp(a)& \text{($\spp$ is equivariant)}\\
& = b \wedge \upsilon(a1_Q)\;. \qed
\end{align*}
\end{proof}

Let $(Q,\spp,\upsilon)$ be an equivariantly supported reflexive $A$-$A$-quantal frame. Let us define locales $G_0$ and $G_1$ as follows:
\[
\opens(G_0)=A\;,\quad  \opens(G_1)=Q\;.
\] 

\begin{lemma}\label{open}
$\spp:Q\to A$ is the direct image $d_!$ of an open map $d:G_1\to G_0$.
\end{lemma}
\begin{proof}
Recall from Lemma~\ref{adjoint} that the map $\lres{(.)}{1_Q}:A\to Q$ is the right adjoint of $\spp$ and therefore it is a frame homomorphism that defines a map of locales $d:G_1\to G_0$ which is semiopen with $\spp=d_!$ because $d^*$ is the right adjoint $\spp_*$ of $\spp$:
\begin{align*}
\spp(d^*(a)) &= \spp(\lres{a}{1_Q}) = a& \text{for all $a\in A$ (because $\spp$ is equivariant)\;, }\\
d^*(\spp(q)) &= \lres{\spp(q)}{1_Q}\geq q& \text{for all $q\in Q$\;.}
 \end{align*}
Now, let us check that $d$ satisfies the Frobenius reciprocity condition in order to show that $d$ is open. Let $q\in Q$ and $a\in A$:
\begin{align*}
d_!(d^*(b)\wedge a) &= d_!((\lres{b}{1_Q})\wedge a)\\
& =  d_!(\lres{b}{(1_Q\wedge a)})& \text{[by \eqref{qf1}]} \\
& = d_!(\lres{b}{a}) \\
& = \spp(\lres{b}{a}) = b\wedge \spp(a)& \text{(because $\spp$ is equivariant)}\\
& = b\wedge d_!(a)\;. \qed
\end{align*} 
\end{proof}
We can now define a locale map $i:G_1\to G_1$ by the condition $i^*(q)=q^*$ because the involution of $Q$ is a frame isomorphism. Moreover, we have $i\circ i=id_Q$ and $i_!=i^*$. Let us define an open map 
\[  r:G_1\to G_0   \]
by putting $r=d\circ i$\;.

\begin{lemma}
Let $(Q,\spp,\upsilon)$ be an equivariantly supported reflexive $A$-$A$-quantal frame. The tensor product $Q\otimes_{A} Q$ coincides with the pushout of the homomorphisms $d^*$ and $r^*$ in $\Frm$.
\end{lemma}

\begin{proof}
To prove that
\[ \xymatrix{
A \ar^{d^*}[rr]\ar_{r^*}[d] && Q \ar^{\pi^*_2}[d] \\ 
Q \ar_{\pi^*_1}[rr]     && Q\otimes_{A} Q } \]
is a pushout in $\Frm$, we shall show that for all $b,c\in Q$ and $a\in A$ the equality
\[
(b\wedge r^*(a))\otimes c = b\otimes (d^*(a)\wedge c)
\]
is equivalent to 
\[
b\otimes (\lres{a}{c})=(\rres{b}{a})\otimes c\;,
\] 
\end{proof}
which is a consequence of the following two derivations: 
\begin{align*}
d^*(a)\wedge c &= (\lres{a}{1_Q})\wedge c = \lres{a}{c}\;;\\
b\wedge r^*(a) &= b\wedge (\lres{a}{1_Q})^* \\
& = b\wedge (\rres{1_Q}{a}) = \rres{b}{a}\;. \qed
\end{align*}

Our candidate for the inclusion of units $u:G_0\to G_1$ will be $u^*(a)=\upsilon(a)$, for all $a\in Q$, as we shall see now.

\begin{lemma}\label{units}
Let $(Q,\upsilon)$ be a reflexive quantal frame. Then the map $u$ defined by $u^*=\upsilon$ satisfies $d\circ u=\ident_{G_0}$ and $r\circ u=\ident_{G_0}$.
\end{lemma} 

\begin{proof}
The frame homomorphism $\upsilon:Q\to A$ defines a map of locales $u:G_0\to G_1$ given by $u^*=\upsilon$. For all $a\in A$ we have
\[
u^*(d^*(a))=\upsilon(\lres{a}{1_Q})=a\;,
\]
so $d\circ u=\ident_{G_0}$.
Similarly, 
\[
u^*(r^*(a))=\upsilon((\lres{a}{1_Q})^*)=\upsilon(\rres{1_Q}{a})=a\;,
\]
so $r\circ u=\ident_{G_0}$. \qed
\end{proof}

\begin{corollary}
Let $(Q,\spp,\upsilon)$ be an equivariantly supported reflexive quantal frame, and let $G$ be its associated involutive localic graph:
\[
\xymatrix{
G=& G_1\ar@(ru,lu)[]_i\ar@<1.2ex>[rr]^r\ar@<-1.2ex>[rr]_d&&G_0\,.\ar[ll]|u
}
\]
Then $G$ is an involutive reflexive open graph.
\end{corollary}


\subsection{Multiplicative quantal frames}

\begin{lemma}
Let $Q$ be an $A$-$A$-quantal frame. The quantale multiplication has the following factorization in the category of sup-lattices:
\[
\xymatrix{
 Q\otimes Q \ar@{->>}[d]_{\pi}  \ar[dr]^{\mu} \\
 Q\otimes_{A} Q  \ar[r]_-{\mu_A} & Q
}
\]
\end{lemma}

\begin{proof}
By definition $\mu$ preserves joins in each variable, and even more it is middle-linear because
\[ \mu(a\otimes \lres{c}{b})=a(\lres{c}{b})=(\rres{a}{c})b=\mu(\rres{a}{c}\otimes b)  \]
for all $a,b\in Q$ and $c\in A$. The factorization now follows from the definition of the tensor product. \qed
\end{proof}

We adopt similarly terminology to that of \cite{Re07}:

\begin{definition}
The homomorphism $\mu_A:Q\otimes_{A} Q\to Q$ in the above factorization will be referred to as the \emph{reduced multiplication} of $Q$.
By a \emph{multiplicative quantal frame} is meant an $A$-$A$-quantal frame such that the right adjoint of the reduced multiplication preserves joins.
\end{definition}

\begin{example}\label{iqfmultiplicativity}
Every inverse quantal frame $Q$ is multiplicative because $(\mu_A)_*$ is given by:
\begin{equation*}
(\mu_A)_*(q)=\V_{u\in \ipi(Q)} u\otimes u^*q
\end{equation*}
for all $q\in Q$, so it clearly preserves joins (\cf\ \cite{GSQS}*{Lemma 3.15}).
\end{example}

\begin{theorem}\label{principal, theo: semicategory}
Let $(Q,\spp)$ be a multiplicative equivariantly supported $A$-$A$-quantal frame. The localic graph 
\[
\xymatrix{
G=G_2\ar[r]^-{m}& G_1\ar@(ru,lu)[]_i\ar@<0.5ex>[rr]^r\ar@<-0.5ex>[rr]_d&&G_0\;,
}
\]
where $m$ is defined by
\begin{equation*}
m^*(a)=(\mu_A)_*(a)= \V_{xy\leq a} x\otimes y\,,
\end{equation*} 
is an involutive open semicategory.
\end{theorem}

\begin{proof}
The proof of the associativity of $m$ is completely analogous to the proof of associativity in \cite{Re07}*{Th.\ 4.8}. The proof that $i$ is an involution for $m$  is the same as the one given in \cite{PR12}*{Lemma 2.16}. Therefore, the only thing left to prove is that the following diagrams are commutative:

\[ \xymatrix{
G_2 \ar^{\pi_1}[rr]\ar_{m}[d] && G_1 \ar^{d}[d] \\ 
G_1 \ar_{d}[rr]     && G_0 }\quad \xymatrix{
G_2 \ar^{\pi_2}[rr]\ar_{m}[d] && G_1 \ar^{r}[d] \\ 
G_1 \ar_{r}[rr]     && G_0\,. }
\]
In order to verify the equation $d\circ m=d\circ \pi_1$ let us use the inverse image homomorphisms. For each $z\in A$ we see that $\pi^*_1(d^*(z))=(\lres{z}{1_Q})\otimes 1_Q\leq m^*(d^*(z))$ by taking $x=\lres{z}{1_Q}$ and $y=1_Q$ in 
\[
m^*(d^*(z))=\V_{xy\leq \lres{z}{1_Q} } x\otimes y\;.
\]
For the converse inequality we have
\begin{align*}
m^*(d^*(z)) & =\V_{xy\leq \lres{z}{1_Q} } x\otimes y \\
 & =\V_{xy\leq \lres{z}{1_Q} } x\otimes (\lres{\spp(y)}{y}) \\
& =\V_{xy\leq \lres{z}{1_Q} } (\rres{x}{\spp(y)})\otimes y\\
& \leq \V_{xy\leq \lres{z}{1_Q} } xyy^*\otimes y\\
& \leq \V_{xy\leq \lres{z}{1_Q} } xy1_Q\otimes 1_Q\\
& \leq (\V_{xy\leq \lres{z}{1_Q} } xy)1_Q\otimes 1_Q\\
& \leq (\lres{z}{1_Q})1_Q\otimes 1_Q \\
& = \lres{z}{1_Q}\otimes 1_Q = \pi^*_1(d^*(z))\;. 
\end{align*} 
The condition $r\circ m=r\circ \pi_2$ is proved similarly. \qed
\end{proof}


\section{Open groupoids}\label{principal, section: opengroupoids}

\subsection{Unit laws}

\begin{definition}\label{principal, def: unitlaws}
Let $Q$ be a multiplicative equivariantly supported reflexive quantal frame $(Q,\spp,\upsilon)$. We say that $Q$ satisfies \emph{unit laws} if  for all $a\in Q$ the following condition holds:
\begin{equation}\label{unitlaws}
\V_{xy\leq a} (\lres{\upsilon(x)}{y}) = a\,.
\end{equation}
\end{definition}

\begin{lemma}\label{opencategory} 
Let $Q$ be a multiplicative equivariantly supported reflexive quantal frame that satisfies unit laws, and let $G$ be its associated involutive localic graph:
\[
\xymatrix{
G=G_2\ar[r]^-{m}& G_1\ar@(ru,lu)[]_i\ar@<1.2ex>[rr]^r\ar@<-1.2ex>[rr]_d&&G_0\,.\ar[ll]|u
}
\]
Then $G$ is an open involutive category.
\end{lemma}

\begin{proof}
We already know from Theorem~\ref{principal, theo: semicategory} that $G$ is an open involutive semicategory. Now we prove the unit laws of an internal category, as illustrated by the following commutative diagram:
\[
\xymatrix{
G_0\times_{G_0}G_1\ar[rr]^{u\times \ident}&&G_1\times_{G_0}G_1\ar[d]_m&&G_0\times_{G_0}G_1\ar[ll]_{\ident\times u} \\
G_1\ar[u]^{\langle d, \ident \rangle}&&G_1\ar@{=}[ll]\ar@{=}[rr]&&G_1\;. \ar[u]_{\langle \ident, r \rangle}
}
\]
The commutativity of the left hand square can be proved in terms of inverse images. For all $a\in Q$ we have:
\begin{align*}
[d^*,\ident]\circ (u^*\otimes \ident)\circ m^*(a) & = [d^*,\ident]\circ (u^*\otimes \ident)(\V_{xy\leq a}x\otimes y)\\
& = \V_{xy\leq a} d^*(u^*(x)) \wedge y\\
& = \V_{xy\leq a} (\lres{\upsilon(x)}{1_Q}) \wedge y\\
& = \V_{xy\leq a} (\lres{\upsilon(x)}{(1_Q\wedge y)})& \text{[by \eqref{qf1}]}\\
& = \V_{xy\leq a} (\lres{\upsilon(x)}{y})\\
& = a\;.& \text{[due to (\ref{unitlaws})]}
\end{align*} 
The commutativity of the right hand square follows from the left one using the involution laws $d\circ i=r$, $i\circ i=\ident$ and $i\circ u=u$:
\begin{align*}
m\circ (\ident\times u)\circ \langle \ident,r \rangle &= m\circ \langle \ident,u\circ r \rangle\\
&= m\circ \langle i\circ i,i\circ u\circ d\circ i \rangle\\
&= m\circ (i\times i)\circ \langle \ident,u\circ d \rangle \circ i\\
&= i\circ m\circ\langle \ident,u\circ d \rangle \circ i\\
&= i\circ \ident\circ i\\
&= \ident\;. \qed
\end{align*}
\end{proof}

\begin{remark}
In the unital case the inclusion of units map $u:G_0\to G_1$ is given by $u^*(a)=a\wedge e$. This is an open map of locales because the frame homomorphism $u^*(a)=a\wedge e$ is the right adjoint of the sup-lattice inclusion $\iota: \spp(Q)\to Q$, whose direct image is $u_!=\iota$. Therefore it is possible to prove the unit laws of an internal category in terms of direct images (\cf\ \cite{Re07}*{Th. 4.8}) without appealing to the \emph{unit laws} axiom. We remark that in the non-unital setting the map $u:G_0\to G_1$ defined as above is not necessarily open, therefore the unit laws are required.   
\end{remark}


\subsection{Groupoid quantales}

\begin{definition}\label{principal, def: groupoidquantale}
By a \emph{groupoid quantale} $Q$ will be meant a multiplicative equivariantly supported reflexive quantal frame that satisfies unit laws and moreover satisfies the following condition, referred to as the \emph{inverse law}, for all $a\in Q$:
\begin{equation}\label{inverselaws}
\lres{\upsilon(a)}{1_Q} = \V_{xy^*\leq a} x\wedge y\;.
\end{equation}  
\end{definition}

\begin{remark}\label{remark:inverselaws1}
We remark that \eqref{inverselaws} can be written as
\begin{equation*}
\lres{\upsilon(a)}{1_Q} = \V_{xx^*\leq a} x\;.
\end{equation*}  
In fact, 
\[  \V_{xx^*\leq a} x\le \V_{xy^*\leq a} x\wedge y\leq \V_{(x\wedge y)(x\wedge y)^*\leq a} x\wedge y= \V_{zz^*\leq a} z\;.   \]
\end{remark}

\begin{theorem}\label{opengroupoid}
Let $Q$ be a groupoid quantale, and let $G$ be its associated involutive localic graph:
\[
\xymatrix{
G=G_2\ar[r]^-{m}& G_1\ar@(ru,lu)[]_i\ar@<1.2ex>[rr]^r\ar@<-1.2ex>[rr]_d&&G_0.\ar[ll]|u
}
\]
Then $G$ is an open groupoid.
\end{theorem}

\begin{proof}
Due to Lemma~\ref{opencategory} all there is left to do is prove that the involution $i$ satisfies the inverse laws of an internal groupoid, which are described by the commutativity of the following diagram:
\[
\xymatrix{
G_1\ar[d]_{d}  \ar[rr]^{\langle \ident, i \rangle}&&G_2\ar[d]_m&&G_1\ar[d]^{r}\ar[ll]_{\langle i, \ident \rangle} \\
G_0\ar[rr]_{u}&&G_1&&G_0\ar[ll]^{u}
}
\]
Using inverse image homomorphism we shall prove the commutativity of the left hand square. For all $a\in Q$ we have
\begin{align*}
[\ident,i^*]\circ m^*(a) &= [\ident,i](\V_{xy\leq a} x\otimes y)\\
& = \V_{xy\leq a} x\wedge y^*\\
& = \V_{xy^*\leq a} x\wedge y\\
& = \lres{\upsilon(a)}{1_Q}\\
& = d^*(u^*(a))\;. 
\end{align*}
The commutativity of the right hand square follows from the involution and the commutativity of the left square. Therefore $G$ is a groupoid, and it is open due to Lemma~\ref{open}. \qed
\end{proof}

\begin{remark}
Once again regarding the unital case, a unital stably supported quantal frame $Q$ satisfies inverse laws if and only if it is an inverse quantal frame, \ie, $\V \ipi(Q)=1_Q$ (\cf\ \cite{Re07}*{Lemma 4.18}). As we shall see, in the class of unital equivariantly supported reflexive quantal frames the axiom of unit laws can be derived from the axiom of inverse laws (\cf\ Corollary~\ref{unitimpliesinverse} below). Hence, for unital quantales these two axioms are not independent.      
\end{remark}

The following examples show that unit laws and inverse laws axioms are independent in general:

\begin{example}
Consider the unital quantal frame $Q=\{0,e,1\}$ with trivial involution and the obvious multiplication table (the multiplicative unit is $e$). This is stably supported with  $\spp(a)=e$ for all $a\neq 0$. It is also a multiplicative quantal frame because $Q$ is totally ordered, so that the right adjoint
\[
\mu_*:Q\to Q\otimes_A Q=Q\otimes Q
\]
of the reduced multiplication $\mu$ necessarily preserves joins. However, it is not an inverse quantal frame because $1$ is not a join of partial units (\cf\ \cite{Re07}*{Example 4.22}). Let $A$ be the locale $\{0,e\}$. Then $Q$ is a multiplicative equivariantly supported $A$-$A$-quantal frame, and it is reflexive with $\upsilon=\spp$. It is easy to check that it satisfies unit laws, but it does not satisfy inverse laws (otherwise it would have to be an inverse quantal frame --- see Theorem~\ref{theorem:iqf1} below), as the following shows:
\[
\lres{\upsilon(e)}{1} = 1\neq e=\V_{xx^*\le e} x\;.
\]
\end{example}

\begin{example}
Let $A$ be the locale $\{0_A,1_A\}$, and consider the (non-unital) equivariantly supported reflexive $A$-$A$-quantal frame $Q=\{0,a,1\}$ such that $a^2=1$, $a^*=a$, and $\upsilon(a)=0_A$. Note that, necessarily, $\spp(a)=\spp(1)=\upsilon(a1)=1_A$. Similarly to the previous example, $Q$ is multiplicative because it is totally ordered. Now $Q$ satisfies the inverse laws:
\begin{align*}
\lres{\upsilon(0)}{1} &= 0 = \V_{xx^*\le 0} x\\
\lres{\upsilon(a)}{1} &= 0 = \V_{xx^*\le a} x\\
\lres{\upsilon(1)}{1} &= 1 = \V_{xx^*\le 1} x\;.
\end{align*}
But $Q$ does not satisfy the unit laws because
\[
\V_{zw\le a} (\lres{\upsilon(z)}{w})=0 \neq a\,.
\]
\end{example}


\subsection{Unital groupoid quantales}

\begin{theorem}\label{theorem:iqf1}
The class of unital equivariantly supported reflexive quantal frames satisfying inverse laws can be identified with the class of inverse quantal frames.  
\end{theorem}

\begin{proof}
Let us consider $(Q,\spp,\upsilon,e)$ an equivariantly supported reflexive quantal frame with a unit and satisfying inverse laws. Recall from Corollary~\ref{stably1} that $Q$ is a stable quantal frame with 
\[  \spp(q)=q1_Q\wedge e\,. \]
Now we show that $\upsilon(q)=q\wedge e$. In fact, by applying the equivariance of the support (\cf\ Definition~\ref{inverselaws}) we get
\begin{align*}
\upsilon(q) &= \spp(\lres{\upsilon(q)}{1_Q})\\
&= \spp(\V_{xx^*\leq q} x)& \text{(by Remark~\ref{remark:inverselaws1})}\\
&= \V_{xx^*\leq q}\spp(x)\\
&= \V_{xx^*\leq q}x1_Q\wedge e\\
&= \V_{xx^*\leq q}xx^*\wedge e\\
&= q\wedge e\;.
\end{align*}  
Now, due to the involution and because $Q$ satisfies inverse laws, we have  $\V \ipi(Q)=1_Q$ (\cf\ \cite{Re07}*{Lemma 4.18}). This implies that $Q$ is an inverse quantal frame. The converse follows from Corollary~\ref{stably1} and \cite{Re07}*{Lemma 4.18}. \qed
\end{proof}

\begin{theorem}\label{invqufrsatunitlaws}
Let $Q$ be an inverse quantal frame with base locale $A=\downsegment(e)$. Then, regarded as an equivariantly supported reflexive $A$-$A$-quantal frame, $Q$ satisfies unit laws.
\end{theorem}

\begin{proof}
On one hand we have
\begin{align*}
\V_{xy\leq q} (\lres{\upsilon(x)}{y}) & = \V_{xy\leq q} (x\wedge e)1_Q \wedge y\\
& = \V_{xy\leq q} \spp(x\wedge e)y\\
& = \V_{xy\leq q} (x\wedge e)y& \text{(because $x\wedge e\in \downsegment(e)$)}\\
& \leq \V_{xy\leq q} xy\\
& = q\;.     
\end{align*}
This shows that $\V_{xy\leq q} (\lres{\upsilon(x)}{y})\leq q$. In order to prove that $\V_{xy\leq q} (\lres{\upsilon(x)}{y})\geq q$, we show that $\V_{xy\leq q} (\lres{\upsilon(x)}{y})\geq s$ for all partial units $s\in \ipi(Q)$ such that $s\leq q$. Then, taking the supremum and using the fact that this supremum is $q$ because $Q$ is inverse, we get the inequality. Let $s\in \ipi(Q)$ such that $s\leq q$. We have $s=ss^*s\leq q$. Therefore taking $x=ss^*$ and $y=s$, we get
\begin{align*}
\V_{xy\leq q} (\lres{\upsilon(x)}{y}) & \geq \lres{\upsilon(ss^*)}{s} \\
& = (ss^*\wedge e)1_Q \wedge s\\
& = \spp(s)1_Q \wedge s\\
& = \spp(s)s\\
& = s\;. \qed
\end{align*} 
\end{proof}

\begin{corollary}\label{unitimpliesinverse}
Let $A$ be a locale and $Q$ a unital equivariantly supported reflexive $A$-$A$-quantal frame. If $Q$ satisfies inverse laws then it satisfies unit laws.
\end{corollary}

\begin{proof}
Assume that $Q$ satisfies inverse laws. Theorem~\ref{theorem:iqf1} implies that $Q$ is an inverse quantal frame, and thus, by Theorem~\ref{invqufrsatunitlaws}, it satisfies unit laws. \qed
\end{proof}


\subsection{Quantal groupoids}

Given a groupoid quantale $Q$ we denote its associated open groupoid by $\mathcal{G}(Q)$. Recall that by a \emph{quantal groupoid} is meant a localic groupoid $G$ whose multiplication map is semiopen. Now we shall see that if $G$ is an open groupoid then the associated quantale $\opens(G)$ necessarily is a groupoid quantale.

\begin{theorem}\label{principal, theo: groupoidquantale}
Let $G$ be an open groupoid. Then its associated quantale $\opens(G)$ is a groupoid quantale.
\end{theorem}

\begin{proof}
Define $\upsilon=u^*$ and $\spp=d_!$. From the proof of \cite{PR12}*{Th.\ 2.41} it follows that $\opens(G)$ is a multiplicative reflexive quantal frame satisfying satisfying \eqref{unitlaws} [note that the meet $\upsilon(x) \wedge y$ in \cite{PR12}*{Th.\ 2.41} becomes our restriction $\lres{\upsilon(x)}y$ because the base locale in \cite{PR12} is $\rs(\opens(G))$]. In addition, in \cite{Re07}*{Lemma 5.4} it is seen that $\opens(G)$ satisfies \eqref{inverselaws}, which is written as follows:
\begin{equation}
d^*(u^*(q)) = \V_{bc^*\leq q} b\wedge c=\V_{bb^*\leq q} b\;.
\end{equation}
Since for quantal groupoids $d^*(q)$ is always right-sided (similarly, $r^*(q)$ is left-sided) \cite{Re07}*{Lemma 5.3}, in particular $d^*(u^*(q))$ is always right-sided and thus 
\begin{equation}\label{eq:duq1=q1}
d^*(u^*(q1))=q1\;. 
\end{equation}
Moreover, for all $a,b\in \opens(Q)$, we have
\[
a\leq \V_{xx^*\leq aa^*} x=d^*(u^*(aa^*))
\]
and 
\[
d^*(u^*(a))\wedge b\leq \V_{xx^*\leq a} x\wedge b\leq \V_{xx^*b\leq ab} x\wedge b\leq \V_{(x\wedge b)((x\wedge b))^*(x\wedge b)\leq ab} x\wedge b\leq  ab\;.
\]
Therefore
\begin{equation}\label{eq:x1yxxy}
a1\wedge b\leq aa^*b
\end{equation}
for all $a,b\in \opens(G)$, and all we have to do is show that $\opens(G)$ is an equivariantly supported $\opens(G_0)$-$\opens(G_0)$-quantal frame. First, note that $\opens(G_1)$ can be regarded as an $\opens(G_0)$-$\opens(G_0)$-quantale with actions defined by, for all $a\in \opens(G_0)$ and $q\in \opens(G_1)$,
\begin{align*}
\lres{a}{q} &= d^*(a)\wedge q\\
\rres{q}{a} &= r^*(a)\wedge q\;,
\end{align*}
and the involution defined by
\[ a^*=i_!(a)=i^*(a)\;.  \] 
Now it is straightforward to see that $\opens(G_1)$ is an $\opens(G_0)$-$\opens(G_0)$-quantal frame. Taking into account that $G$ is an open groupoid, we can define a sup-lattice homomorphism $\spp: \opens(G_1)\to \opens(G_0)$ given by $\spp:=d_!$. In order to see that this is the support of $\opens(G_1)$ we notice that, by \cite{Re07}*{Lemma 5.5},
\begin{equation}
d_!(x) =u^*(x1)
\end{equation}
for all $x\in\opens(G_1)$.  
Then for all $x,y\in \opens(G_1)$ we have
\begin{itemize}
\item $\spp(1)=d_!(1)=u^*(11)=u^*(1)=1_{\opens(G_0)}$ because $u^*$ is a surjective frame homomorphism;
\item  $\lres{\spp(x)}{x}=d^*(d_!(x))\wedge x=d^*(u^*(x1))\wedge x= x1\wedge x= x$ ---
similarly we prove $x=\rres{x}{\spp(x^*)}$\;;
\item $\lres{\spp(x)}{y} = d^*(d_!(x)) \wedge y= d^*(u^*(x1)) \wedge y= x1\wedge y\leq xx^*y$, using \eqref{eq:duq1=q1} and \eqref{eq:x1yxxy}.  
\end{itemize}
It remains to show that $\spp=d_!$ is $A$-equivariant, which follows easily from the fact that $d$ is open and hence satisfies the Frobenius reciprocity condition:
\begin{align*}
\spp(\lres{a}{q}) &= d_!(d^*(a)\wedge q)\\
& = a\wedge d_!(q)\\
& = a\wedge \spp(q)\;.
\end{align*}  
This shows that $\opens(G_1)$ is a multiplicative equivariantly supported reflexive  $\opens(G_0)$-$\opens(G_0)$-quantale frame which in addition satisfies the unit and inverse laws, so $\opens(G)$ is a groupoid quantale. \qed
\end{proof}

Now the following is straightforward:

\begin{theorem}\label{bijection}
$\mathcal{G}(\opens(G))\cong G$ and $\opens(\mathcal{G}(Q))\cong Q$ for any open groupoid $G$ and a groupoid quantale $Q$.
\end{theorem}


\subsection{The pair groupoid}\label{pairgroupoid}

Let $G$ be an open groupoid and let us denote by $\intg G$ the \emph{pair groupoid} of $G$, in other words the pullback $\intg G_1=G_1\times_{G_0} G_1$ equipped with the usual groupoid structure:
\[\xymatrix{
\intg G_2\ar[rr]^-{\intg m}&&\intg G_1\ar@(ru,lu)[]_{\intg \imath}\ar@<1.2ex>[rrr]^-{\intg d=\pi_1}\ar@<-1.2ex>[rrr]_-{\intg r=\pi_2}&&&\intg G_0=G_1\ar[lll]|-{\intg u=\Delta}}
\]
where in particular $\intg G_2$ is the pullback of $\intg d$ and $\intg r$.
\begin{theorem}\label{principal, theo: pairgrpd}
Let $Q$ be a groupoid quantale. Then $Q\otimes_A Q$ is a groupoid quantale which is exactly the quantale of the pair groupoid $G_1\times_{G_0} G_1 \rightrightarrows G_1$, where $\opens({G_1})=Q$ and $\opens({G_0})=A$.
\end{theorem}
\begin{proof}
To start with, let us notice that pure tensors of $Q\otimes_A Q$ can be written in the form $x\otimes y$ with $\spp(x)=\spp(y)$. Indeed, $Q\otimes_A Q$ is the pushout of $d^*$ and $r^*$ in $\Frm$, and thus we have, for all $a\in A$,
\[\lres{a}{x}\otimes y = (d^*(a)\wedge x)\otimes y=x\otimes(d^*(a)\wedge y)=x\otimes \lres{a}{y}\;.\]
Hence, $x\otimes y=\lres{\spp(x)}{x}\otimes\lres{\spp(y)}{y}=\lres{\spp(y)}{x}\otimes\lres{\spp(x)}{y}=w\otimes z$ with $\spp(z)=\spp(w)=\spp(x)\spp(y)$. Now, notice that $Q\otimes_A Q$ is an involutive $Q$-$Q$-quantal frame with
\begin{align}
\lres{z}{(x\otimes y)}&=(x\wedge z)\otimes y\\
\rres{(x\otimes y)}{z}&= x\otimes (y\wedge z)\\
(x\otimes y)^*&= y\otimes x\\
(x\otimes y)(z\otimes w)&= (\lres{\spp(y\wedge z)}{x})\otimes w
\end{align}
for all $x,y,z\in Q$. Moreover, it can be endowed with a support $\intg \spp:Q\otimes_A Q\to Q $ given by the sup-lattice homomorphism
\begin{align}
\intg \spp(x\otimes y) &= \lres{\spp(y)}{x}
\end{align}
and with a frame homomorphism $\intg \upsilon: Q\otimes_A Q\to Q$ given by
\begin{align}
\intg \upsilon(x\otimes y) &=x\wedge y\,.
\end{align} 
It is straightforward to verify that $(Q\otimes_A Q,\intg \spp, \intg \upsilon)$ is an equivariantly supported $Q$-$Q$-quantal frame. By Theorem~\ref{bijection} we conclude that $Q\otimes_A Q$ is exactly the quantale of the pair groupoid $G_1\times_{G_0} G_1 \rightrightarrows G_1$, where $\opens({G_1})=Q$ and $\opens({G_0})=A$.\qed
\end{proof}


\subsection{Effective equivalence relations}\label{subsec: principal groupoids}

Recall that by an \emph{effective equivalence relation} is meant an open groupoid $G$ which is the kernel pair of the co-equalizer of $d$ and $r$.  In other words, this means that the square
\[
\vcenter{\xymatrix{
G_1\ar[r]^-{r} \ar[d] \ar[d]_{d}  &G_0 \ar@{->>}[d]\\
G_0\ar@{->>}[r]& G_0/G\;,
}
}
\]
where $G_0/G$ is the \emph{orbit locale}, is a pull-back in $\Loc$. 

\begin{theorem}\label{principal, theo: effectiveeqr}
Let $Q$ be a groupoid quantale and let $G=\mathcal{G}(Q)$ be its open groupoid. Then $Q$ is a principal quantale if and only if $G$ is an effective equivalence relation.
\end{theorem}

\begin{proof}
Suppose that $Q$ is a principal quantale. By Theorem~\ref{thm:effective}\eqref{thm:effective3}, the triple $(Q,d^*,r^*)$ is the co-kernel pair of the frame inclusion $i: E\to A$ where
\[
E=\{ b\in A\mid d^*(b)=r^*(b)  \}\;.
\]
Therefore by Theorem~\ref{opengroupoid} the triple $(G_1,d,r)$ is the kernel pair of the co-equalizer of $d$ and $r$,
\[
\xymatrix{
G_1 \ar@<0.5ex>[rr]^{d}\ar@<-0.5ex>[rr]_{r} && G_0\ar@{->>}[r]^-{q}   & G_0/G\;,
}
\]
so $G$ is an effective equivalence relation. The converse follows immediately from Theorem~\ref{bijection}. \qed
\end{proof}

\begin{remark}
This provides a justification for the name \emph{principal quantales} because effective equivalence relations correspond closely to principal 
bundles on groupoids (see, \eg, \cites{QuijanoPhD,KoMoer}). In fact the terminology \emph{principal groupoids} is even used for not necessarily effective equivalence relations --- see, \eg\ \cite{Renault}.
\end{remark}

The notion of principal groupoid quantale also yields, in a simplified situation, the first quantale description of an \'etale-complete groupoid (\cf\ \cites{Moer88,Moer90,Bunge}), as the next result shows:

\begin{corollary}\label{cor: etalecompleteness}
Let $Q$ be a principal groupoid quantale and let $G=\mathcal{G}(Q)$ be its open groupoid. Then the topos $BG$ is localic, and $G$ is \'etale-complete.
\end{corollary}
\begin{proof}
Suppose that $Q$ is a principal groupoid quantale. Then, by Theorem~\ref{principal, theo: effectiveeqr}, $G$ is an effective equivalence relation and therefore satisfies all the asumptions of \cite{KoMoer}*{Prop.\ 3.3}, which implies that the topos $BG$ is localic and $G$ is \'etale-complete. \qed
\end{proof}

\end{document}